\documentclass[centertags, reqno]{amsart}                  
\usepackage{amssymb}  
\usepackage{graphicx}
\usepackage{url}    
\usepackage{dsfont}

\newcommand{\id}{\mathds{1}}

\newtheorem{thm}{Theorem}
\newtheorem*{thm*}{Theorem} 
\newtheorem{lemma}{Lemma}
\newtheorem*{lem*}{Lemma}
\newtheorem{cor}{Corollary}

\newtheorem*{convention}{Convention}

\theoremstyle{definition}

\theoremstyle{remark}
\newtheorem{rem}{Remark} 
\newtheorem{ex}{Example}

\newcommand{\mr}{{\mathbb R}}

\newcommand{\mn}{{\mathbb N}}

\newcommand{\mc}{{\mathbb C}}
\newcommand{\md}{{\mathbb D}}

\renewcommand{\rho}{\varrho}
\newcommand{\eps}{\varepsilon}

\newcommand{\ran}{\operatorname{Ran}}
\newcommand{\hil}{\mathcal{H}}
\newcommand{\bdd}{\mathcal{B}}
\newcommand{\num}{\operatorname{Num}}

\usepackage{xcolor}
\usepackage{enumerate} 
 
\usepackage{mathtools}

\newcommand{\VERT}[1]{{\left\vert\kern-0.25ex\left\vert\kern-0.25ex\left\vert #1 
    \right\vert\kern-0.25ex\right\vert\kern-0.25ex\right\vert}}

\begin{document}

\title{Eigenvalues of compactly perturbed operators via entropy numbers}

\author[M. Hansmann]{Marcel Hansmann}
\address{Faculty of Mathematics\\ 
Chemnitz University of Technology\\
Chemnitz\\
Germany.}
\email{marcel.hansmann@mathematik.tu-chemnitz.de}

\thanks{I would like to thank M. Demuth and G. Katriel for some valuable remarks. This work was funded by the Deutsche Forschungsgemeinschaft (DFG, German Research Foundation) - Project number HA 7732/2-1.}

\begin{abstract}
We derive new estimates for the number of discrete eigenvalues of compactly perturbed operators on Banach spaces, assuming that the perturbing operator is an element of a weak entropy number ideal. Our results improve upon earlier results by the author and by Demuth et al. The main tool in our proofs is an inequality of Carl. In particular, in contrast to all previous results we do not rely on tools from complex analysis.  
\end{abstract}

\maketitle

\section{Introduction}   

The classical theory of Riesz \cite{MR1555146} tells us that the non-zero spectrum of a compact  operator $K$, acting on an infinite dimensional Banach space, consists of discrete eigenvalues $(\lambda_n)$ and that these eigenvalues can accumulate at $0$, the only point in the essential spectrum, only. Looking at this result it is quite natural to ask for a more quantitative version of it. That is, considering suitable subspaces of compact operators, can we say something about the speed of accumulation of the sequence of eigenvalues? For instance, is this sequence in $l_p$ for a suitable $p$? It turns out that for many classes of compact operators one can study these questions successfully and today the corresponding field of operator theory, the study of the 'distribution of eigenvalues of compact operators', is very well developed and its results are successfully applied in the analysis of concrete (differential) operators. We refer to monographs by K\"onig \cite{MR1863710} and Pietsch \cite{MR917067} for an overview on this topic.

The present paper is \emph{not} about eigenvalues of compact operators. Instead, we will study \emph{compactly perturbed} operators. More specifically, given a 'free' bounded operator $A$ and a compact perturbation $K$, we are interested in the discrete eigenvalues of the perturbed operator $A+K$ and their speed of accumulation to the essential spectrum of $A$ (we recall that by Weyl's theorem, the essential spectra of $A$ and $A+K$ coincide). Our goal is to obtain more precise information on this speed of accumulation, given more restrictive assumptions on the perturbation $K$. As compared to the compact case, which is included as the special case $A=0$, this more general situation is quite less well understood, at least in the case of general Banach space operators which we consider here. 

Unfortunately, we immediately have to admit that our goals, as sketched in the previous paragraph, are a little too high. Indeed, independent of the subspace $\mathcal S$ of compact operators we consider, in the stated generality it is not possible to control the speed of accumulation of $(\lambda_n(A+K))$ in terms of the perturbation $K \in \mathcal S$ alone. To see this, it suffices to consider the trivial case where $K=0$: Then we can always find a suitable $A$ such that the discrete eigenvalues of $A+K=A$ accumulate at the essential spectrum as slowly or as fast as we like. -- Clearly, there exist ways to circumvent this problem and in the present paper we will do it by restricting ourselves to certain subsets of discrete eigenvalues of $A+K$, namely those eigenvalues which lie \emph{outside} the closed disk $\overline\md_{r(A)}:=\{ \lambda \in \mc : |\lambda| \leq r(A)\}$, where $r(A)$ denotes the spectral radius of $A$. 

In order to discuss the results of this paper in a little more detail, we now have to specify the subspace $\mathcal S$ of compact operators we will consider. Here our main emphasis will be on the class $\mathcal S_{p,\infty}^{(e)}$ of all compact operators $K$ whose sequence of \emph{entropy numbers} $(e_n(K))$ is in $l_{p,\infty}$, the weak $l_p$-space (we refer to Section \ref{sec:prelim} for the precise definitions). Since our results look particularly simple in the case where $r(A)=\|A\|$, let us consider this case first: Let $n_{A+K}(s)$ denote the number of discrete eigenvalues of $A+K$ \emph{outside} the disk $\overline\md_s$ (counting algebraic multiplicities). Then we will prove in Theorem~\ref{thm:1} below that for every bounded operator $A$ and $K \in \mathcal S_{p,\infty}^{(e)}, p > 0,$ we have 
\begin{equation}
  \label{eq:2}
  n_{A+K}(s) \leq C_p \frac{s}{(s-\|A\|)^{p+1}} \|K\|_{p,\infty}^p, \qquad s > \|A\|,
\end{equation}
with the explicit constant $C_p=(\ln(2) (p+1)^{p+1})/(2p^p)$. It is instructive to consider a specific consequence of the previous inequality. To this end, let us denote the discrete eigenvalues of $A+K$ outside $\overline{\md}_{\|A\|}$ decreasingly, counting algebraic multiplicities, by $|\lambda_1(A+K)| \geq |\lambda_2(A+K)| \geq \ldots > \|A\|$ (using the convention that  $\lambda_{m+1}(A+K)=\lambda_{m+2}(A+K)=\ldots=\|A\|$ in case there are only $m \in \mn_0$ such eigenvalues). Then the  weak $l_q$-norm of the sequence $(|\lambda_n(A+K)|-\|A\|)$ can be characterized as
$$ \| (|\lambda_n(A+K)| - \|A\|) \|_{{q,\infty}}^q = \sup_{s > \|A\|} \big( (s-\|A\|)^q n_{A+K}(s) \big),$$ 
compare (\ref{eq:1}) below, and so we see that (\ref{eq:2}) leads to the following implications:
\begin{equation}
  \label{eq:3}
  (e_n(K)) \in l_{p,\infty} \quad \Rightarrow \quad (|\lambda_n(A+K)|-\|A\|) \in \left\{
    \begin{array}{cl}
      l_{p,\infty}, & A=0 \\
      l_{p+1,\infty}, & A \neq 0.
    \end{array}\right.
\end{equation}
For the compact case ($A=0$) this is a well-known consequence of a result of Carl \cite[Theorem 4]{MR619953}. In the non-compact case, both (\ref{eq:2}) and (\ref{eq:3}) are new and improve upon earlier results of Hansmann \cite[Corollary 6.6]{Hansmann15} and of Demuth et al. \cite[Corollary 4.3]{MR3296588}. Indeed, in \cite{MR3296588} the authors used infinite determinants and complex function theory to prove (among other things) an inequality similar to (\ref{eq:2}), and hence that $(|\lambda_n(A+K)|-\|A\|) \in l_{p+1,\infty}$, but under the stronger assumption that $K \in \mathcal S_{p}^{(a)}$, i.e. that the \emph{approximation numbers} $(a_n(K))$ are in $l_p$. In \cite{Hansmann15} the author extended the results of \cite{MR3296588} (using the same methods of proof) to the case where $K \in \mathcal S_{p}^{(e)}$, but only for operators on such Banach spaces where the finite rank operators are dense in $\mathcal S_p^{(e)}$ (we note that this need not always be the case, see \cite{MR3282642}, a fact which has been overlooked in \cite{Hansmann15}). Since in general
$$ \mathcal S_p^{(a)} \subset \mathcal S_p^{(e)} \subset \mathcal S_{p,\infty}^{(e)},$$
we see that our above result is stronger than all the previously known results. Our short and simple proof of (\ref{eq:2}) will rely on an inequality of Carl \cite[Theorem 3]{MR619953} between individual eigenvalues and entropy numbers of a compact operator. In particular, our method of proof does not involve the use of complex function theory and is thus completely different to all previous methods.
 
We do not know whether in general Banach spaces the implication (\ref{eq:3}) is optimal or whether the index $p+1$ on the right-hand side can be replaced by something smaller (but greater or equal to $p$).
However, we do know that even for a rank one perturbation $K$ the sequence $(|\lambda_n(A+K)|-\|A\|)$ need not be in $l_{q,\infty}$ for any $q<1$, see Example \ref{ex:1} below.  

On the other hand, for Hilbert space operators the index $p+1$ in (\ref{eq:3}) can indeed be replaced by $p$. This will be a consequence of a more general result for eigenvalues of Hilbert space operators $A+K$ outside the numerical range of $A$, which we present in Theorem \ref{thm:3} below and which extends an earlier result in \cite{Hansmann11}. In this context we will also allow for more general perturbations $K \in \mathcal S_{p,q}^{(e)}$, replacing the weak $l_p$-spaces by general Lorentz sequence spaces $l_{p,q}$.  

Finally, let us take a look at our results for the case where $r(A) < \|A\|$. In this case, using Rota's universal operator model \cite{MR0112040} and the fact that similar operators have the same spectrum, we will be able to derive the following estimate from estimate (\ref{eq:2}) above (see Theorem \ref{thm:2}): If $S \in \mathcal S_{p,\infty}^{(e)}, p > 0$, then
\begin{equation}
  \label{eq:4}
 n_{A+K}(s) \leq 2^{p+1} C_p \frac{ s M_A^p(\frac{s+r(A)}{2})} {(s-r(A))^{p+1}} \|K\|_{{p,\infty}}^p, \qquad s>r(A).  
\end{equation}
As compared to (\ref{eq:2}) here we have an additional term $M_A(r):= \sup_{n \in \mn_0} \|A^n\|/r^n$ on the right-hand side, which can, depending on the properties of the free operator $A$, grow quite rapidly for $r \to r(A)$ (see Remark \ref{rem:2} (iii) below). While we are not aware of any estimate of this kind in the literature, we should note that an estimate of this form can also be obtained using the methods and results of Demuth et al. in \cite{MR3296588} (by exploiting the fact that, as pointed out to the author by Guy Katriel, the function $M_A(r)$ can be used to estimate the resolvent norm of $A$ for resolvent values outside $\overline\md_{r(A)}$). However, our idea to obtain estimate (\ref{eq:4}) via Rota's universal model (and thus avoiding the use of resolvents) seems to be new. 

While in the present paper we will not speak about applications of the above results, we have to mention that all eigenvalue estimates presented can indeed be applied in the study of concrete (differential) operators, since the $\mathcal S_{p,\infty}^{(e)}$-norms of broad classes of operators (for instance of integral or embedding operators) are quite well understood. We refer to the monograph \cite{MR1098497} of Carl and Stephani for an overview on this topic. See also Remark \ref{rem:5} (ii) below.

We conclude this introduction with a short sketch of the contents of this paper: In the following preliminary section we introduce notation and discuss the basics of the Lorentz sequence spaces and the associated $s$-number ideals $\mathcal S_{p,q}^{(s)}$, focusing on the approximation and entropy number ideals used in this paper. In Section 3, we will present and prove our eigenvalue estimates for general Banach space operators and in the final Section 4 we will consider the case of operators on Hilbert spaces.

\section{Preliminaries}\label{sec:prelim}

Throughout this paper $X,Y$ denote complex Banach spaces and $\bdd(X,Y)$ denotes the space of all bounded linear operators from $X$ to $Y$. We set $\bdd(X):=\bdd(X,X)$. If $X$ is a Hilbert space we usually use the symbol $\hil$ instead. The \emph{spectrum} and \emph{discrete spectrum} of $A \in \bdd(X)$ are denoted by $\sigma(A)$ and $\sigma_d(A)$, respectively. We recall that $\sigma_d(A)$ consists of all isolated eigenvalues of finite algebraic multiplicity (also called the \emph{discrete eigenvalues} of $A$). The \emph{essential spectrum} $\sigma_{ess}(A)$ is defined as the set of all $\lambda \in \mc$ such that $\lambda-A$ is not a Fredholm operator. One always has $\sigma_{ess}(A) \cap \sigma_d(A)=\emptyset$ and if $\Omega \subset \mc \setminus \sigma_{ess}(A)$ is a connected component, then either (i) $\Omega \subset \sigma(A)$ or (ii) the spectrum of $A$ in $\Omega$ is purely discrete and the discrete eigenvalues can accumulate at $\sigma_{ess}(A)$ only. In particular, (ii) is true if $\Omega$ is the unbounded component of $\mc \setminus \sigma_{ess}(A)$. The \emph{spectral radius} and \emph{essential spectral radius} of $A$ are defined as
$ r(A)=\max\{ |\lambda| : \lambda \in \sigma(A)\}$ and $r_{ess}(A) = \max\{ |\lambda| : \lambda \in \sigma_{ess}(A)\}$, respectively. The classical \emph{Gelfand formula} tells us that $r(A)=\lim_{n \to \infty} \|A^n\|^{1/n}$. Finally, we note that by \emph{Weyl's theorem} the essential spectra of $A$ and $A+K$ coincide if $K$ is a compact operator. As a general reference for the results mentioned above we refer to \cite{MR1130394}.

In the following subsections we discuss further topics relevant to this paper in a little more detail. As references for this material see, e.g., \cite{MR1098497, MR1863710,MR917067}.  
\subsection{Lorentz sequence spaces}

Let $x=(x_n)_{n=1}^N$, where $N \in \mn \cup \{\infty\},$ denote a complex-valued sequence. We assume that $x \in c_0(\mn)$ (the null sequences) if $N=\infty$. By  $x^*=(x_n^*)_{n=1}^N$ we denote the \emph{decreasing rearrangement} of $x$, i.e. 
$$ x_1^*:= \max_n |x_n|, \quad  x_2^*:= \max_{n \neq m}(|x_n|+|x_m|) - x_1^*, \quad \ldots$$ 
so $x_1^* \geq x_2^* \geq \ldots \geq 0$ and the sets of $x_n^*$ and $|x_n|$ are identical, counting multiplicities. For $0 < p < \infty, 0 < q \leq \infty$ we define
\begin{equation*}
  \|x\|_{p,q}:= \left\{
    \begin{array}{cl}
      \left( \sum\limits_{n=1}^N x_n^{*q} n^{\frac q p  -1} \right)^{\frac 1 q}, & 0 < q < \infty, \\[10pt]
      \sup\limits_{n=1}^N x_n^*n^{\frac 1 p}, & q = \infty,
    \end{array}\right.
\end{equation*}
using the convention that $\infty^{\frac 1 q}=\infty$.
\begin{rem}
  It is not difficult to check that
  \begin{equation}
    \label{eq:1}
\|x\|_{p,\infty}^p= \sup_{r>0} \left( r^p \cdot \#\{n:x_n^* >r\} \right).    
  \end{equation}
\end{rem}
\noindent The \emph{Lorentz sequence spaces} $l_{p,q}= l_{p,q}(\mn)$ are then defined by
\begin{equation*}
  l_{p,q} := \{ x \in c_0(\mn) : \|x\|_{p,q} < \infty\}.
\end{equation*}
We note that $\|.\|_{p,q}$ defines a quasi-norm on $l_{p,q}$ and that $(l_{p,q},\|.\|_{p,q})$ is a quasi-Banach space. In particular, the space $l_{p,p}$ is just the usual space $l_p$ with $\|.\|_{p,p}=\|.\|_p$. We also remark that for $q=\infty$ the space $l_{p,\infty}$ is called \emph{weak} $l_p$-\emph{space}. Finally, we note that the Lorentz sequence spaces are lexicographically ordered, i.e.
\begin{equation}
  \label{eq:8}
  l_{p_1,q_1} \subsetneq l_{p_2,q_2} \text{ if } p_1 < p_2 \text{ or if } p_1=p_2 \text{ and } q_1<q_2.
\end{equation}

\subsection{$s$-number ideals}
Let $\bdd:=\cup_{X,Y} \bdd(X,Y)$ denote the class of all bounded linear operators between Banach spaces. Then a map $s : \bdd \to l_\infty$, assigning to any operator $T \in \bdd$ a bounded sequence $(s_n(T))_{n=1}^\infty$, is called \emph{$s$-number function} if the following holds for all Banach spaces $X_0,X,Y_0,Y$:
\begin{enumerate}
    \item[(s1)] $\|T\|=s_1(T) \geq s_2(T) \geq \ldots \geq 0,\text{ for all } T \in \bdd(X,Y)$,
    \item[(s2)] $s_{n+m-1}(S+T) \leq s_n(S) + s_m(T),\text{ for all } S,T \in \bdd(X,Y), m,n \in \mn,$
    \item[(s3)] $s_n(RST) \leq \|R\| s_n(S) \|T\|$ \\ for all  $R \in \bdd(Y,Y_0),S \in \bdd(X,Y),T \in \bdd(X_0,X), n \in \mn,$
    \item[(s4)] $s_n(T)=0$ if $\operatorname{rank}(T) < n$ and $s_n(\id_{l_2^n})=1$, where $\id_{l_2^n}$ denotes the identity operator on $l_2(\{1,\ldots,n\})$.
\end{enumerate}
The sequence $(s_n(T))_{n=1}^\infty$ is the \emph{sequence of $s$-numbers} of $T$. If only properties (s1)-(s3) are satisfied, then $s$ is called a \emph{pseudo-$s$-number} function (and $(s_n(T))_{n=1}^\infty$ a pseudo-$s$-number sequence). The (pseudo-) $s$-number function $s : \bdd \to l_\infty$ is called \emph{multiplicative} if
\begin{itemize}
    \item[(s5)] $ s_{n+m-1}(ST) \leq s_n(S)s_m(T), \text{ for all } S \in \bdd(X,Y), T \in \bdd(X_0,X), m,n \in \mn.$
\end{itemize}

An example of a multiplicative $s$-number sequence is given by the \emph{approximation numbers}
\begin{equation*} 
 a_n(T):= \inf\{ \|T-F\| : F \in \bdd(X,Y) \text{ with } \operatorname{rank}(F) < n\}, \quad T \in \bdd(X,Y).
\end{equation*}
On Hilbert spaces, this is actually the only $s$-number sequence, i.e. for $T \in \bdd(\hil_1,\hil_2)$ any $s$-number sequence $(s_n(T))_{n=1}^\infty$ coincides with $(a_n(T))_{n=1}^\infty$ (and for compact Hilbert space operators, the sequence $(a_n(T))_{n=1}^\infty$ in turn coincides with the sequence of singular numbers of $T$). On general Banach spaces, however, there are $s$-number sequences different from the approximation numbers, like, e.g., the Weyl-, Gelfand- or Hilbert-number sequences, see \cite{MR917067} for their definitions.

An example of a multiplicative pseudo-$s$-number sequence is provided by the sequence of \emph{entropy numbers} of $T \in \bdd(X,Y)$, defined as
$$ e_n(T) := \inf\left\{ \eps > 0 : \exists y_1,\ldots,y_q \in Y, q \leq 2^{n-1}, \text{s.t. } T(B_X) \subset \cup_{i=1}^q ( \{y_i\} + \eps B_Y)\right\},$$
where $B_X, B_Y$ denote the closed unit balls in $X,Y$. Note that $T$ is compact iff $e_n(T) \to 0$ for $n \to \infty$.

Given any (pseudo-) $s$-number function $s: \bdd \to l_\infty$ and $0 < p < \infty, 0 < q \leq \infty$ we define
$$ \mathcal S_{p,q}^{(s)}(X,Y):= \left\{ T \in \bdd(X,Y) : (s_n(T))_{n=1}^\infty \in l_{p,q} \right\}.$$
This is a linear space which, when equipped with the quasi norm
$$ \|T\|_{p,q}^{(s)}:= \|(s_n(T))\|_{p,q},$$
is a quasi-Banach space. Moreover, it is a Banach \emph{ideal}, meaning that $RST \in \mathcal S_{p,q}^{(s)}(X_0,Y_0)$ if $S \in \mathcal S_{p,q}^{(s)}(X,Y), R \in \bdd(Y,Y_0)$ and $T \in \bdd(X_0,X)$, and in this case $\|RST\|_{p,q}^{(s)} \leq \|R\| \|S\|_{p,q}^{(s)}\|T\|$. The ideals $S_{p,q}^ {(s)}(X,Y)$ are called \emph{$s$-number ideals}. As usual we set $S_{p,q}^ {(s)}(X,X)=:S_{p,q}^ {(s)}(X)$.  Note that from (\ref{eq:8}) we obtain that
\begin{equation}
  \label{eq:10}
    \mathcal S_{p_1,q_1}^{(s)}(X,Y) \subset \mathcal S_{p_2,q_2}^{(s)}(X,Y) \text{ if } p_1 < p_2 \text{ or if } p_1=p_2 \text{ and } q_1<q_2,
\end{equation}
i.e. the $s$-number ideals are lexicographically ordered.
\begin{convention}\label{conv:1}
Since only the approximation number ideals $\mathcal S_{p,q}^ {(a)}$ and entropy-number ideals $\mathcal S_{p,q}^{(e)}$ will play a further role in this article, and since later we will use powers of the corresponding norms (which looks rather awkward with the above standard notation), throughout this article we will set  
\begin{equation*}
  \|K\|_{p,q}^{(e)}=:\|K\|_{p,q}, \qquad \|K\|_{p,q}^{(a)} =: \VERT{K}_{p,q}.
\end{equation*}  
\end{convention}
Concerning the relation between the entropy and approximation number ideals (which both consist of compact operators only)  we note that
\begin{equation*}
  \mathcal S_{p,q}^{(a)}(X,Y) \subset S_{p,q}^{(e)}(X,Y)
\end{equation*}
and that there exists a constant $C({p,q}) \geq 1$ such that for all Banach spaces $X,Y$ and all $T \in \mathcal S_{p,q}^{(a)}(X,Y)$ we have
\begin{equation}
  \label{eq:12}
  \| T\|_{p,q} \leq C({p,q}) \VERT{T}_{p,q}.
\end{equation}
If $X= \hil$ is a Hilbert space, then $S_{p,q}^{(a)}(\hil)=S_{p,q}^{(e)}(\hil)$ (however, be aware that $e_n(T)=a_n(T)$ need certainly not be true).

\section{Eigenvalues of compactly perturbed operators}

Throughout this section we consider a bounded operator $A \in \bdd(X)$ and a compact operator $K$ on $X$. We are interested in the discrete eigenvalues of the perturbed operator $A+K$ \emph{outside} the closed disk $\overline\md_{r_{ess}}=\{ \lambda \in \mc : |\lambda| \leq r_{ess}(A)\}$. \begin{convention}\label{conv2}
In the entire section we will denote the discrete eigenvalues of $A+K$ outside $\overline{\md}_{r_{ess}}$ in decreasing order  by $|\lambda_1| \geq |\lambda_2| \geq \ldots > r_{ess}(A)$ (and with the convention that  $\lambda_{m+1}=\lambda_{m+2}=\ldots=r_{ess}(A)$ if there are only $m$ such eigenvalues). Here each eigenvalue is counted as many times as its algebraic multiplicity. 
\end{convention}

 We set
 $$n_{A+K}(s):= \# \{ n : |\lambda_n| > s\}, \qquad s > r_{ess}(A).$$ 

In order to present our first result, for $p>0$ we need to introduce a function $\Phi_p: (0,1) \to \mr$ by
\begin{equation}
  \label{eq:17}
  \Phi_p(x)= \frac{\left( W\left(\frac 1 p e^{\frac 1 p} x \right) \right)^p}{\left(\frac 1 p - W\left(\frac 1 p e^{\frac 1 p} x \right)\right)^{p+1}x^p},  \qquad x \in (0,1),
\end{equation}
where $W:[0,\infty) \to [0,\infty)$ is the Lambert $W$-function defined by $W(x)e^{W(x)}=x$. Since $W(x) \sim x$ for $x \to 0$, we can continuously extend $\Phi_p$ to $[0,1)$ by setting $\Phi_p(0)=p\cdot e$. As has been shown in \cite{MR3296588} (see the proof of Corollary 4.3 in that article), we also have 
\begin{equation}
  \label{eq:18}
  \Phi_p(x) \leq \frac {(p+1)^{p+1}}{p^p} \frac 1 {(1-x)^{p+1}}, \qquad x \in [0,1).
\end{equation}
\begin{thm}\label{thm:1}
  Let $A \in \bdd(X)$ and $K \in \mathcal S_{p,\infty}^{(e)}(X), p>0$. Then
  \begin{equation}
    \label{eq:15}
 n_{A+K}(s) \leq  \frac{\ln(2)} 2 \frac{\Phi_p\left(\frac{\|A\|}{s}\right)}{s^p} \|K\|_{{p,\infty}}^p, \qquad s>\|A\|.
  \end{equation}
In particular,
\begin{equation}
  \label{eq:16}
 n_{A+K}(s) \leq C_p  \frac{s} {(s-\|A\|)^{p+1}} \|K\|_{{p,\infty}}^p, \qquad s>\|A\|,
\end{equation}
where $ C_p:= (\ln(2)(p+1)^{p+1})/(2p^p)$.
\end{thm}
As explained in the introduction, the previous theorem improves upon results in \cite{Hansmann15} and \cite{MR3296588}. We make this more precise in the following remark.
\begin{rem}
  In \cite[Corollary 6.6]{Hansmann15} it was proven that, with a (more or less explicit) constant $C_p^{(1)}$, it holds that 
  \begin{equation}
    \label{eq:11}
 n_{A+K}(s) \leq {C}_p^{(1)}  \frac{s} {(s-\|A\|)^{p+1}} \|K\|_{p}^p, \qquad s>\|A\|,    
  \end{equation}
if $K \in \mathcal S_p^{(e)}(X)$ is a $\|.\|_p$-limit of finite rank operators. Since $C_p \leq {C}_p^{(1)}$ (as follows from results in \cite{Han17a}) and $\|K\|_{p,\infty} \leq \|K\|_p$ for $K \in \mathcal S_p^{(e)}(X)$, we see that even for perturbations in the smaller class $\mathcal  S_p^{(e)}(X) \subset \mathcal S_{p,\infty}^{(e)}(X)$ estimate (\ref{eq:16}) is stronger than (\ref{eq:11}). Moreover, for perturbations $K$ from the still smaller class $\mathcal S_p^{(a)}(X) \subset \mathcal S_p^{(e)}(X)$ it was proven in \cite[Theorem 4.2]{MR3296588} that, with a constant $C_p^{(2)}$, 
\begin{equation}
  \label{eq:14}
  n_{A+K}(s) \leq  C_p^{(2)} \frac{\Phi_p\left(\frac{\|A\|}{s}\right)}{s^p} \VERT{K}_{p}^p, \qquad s>\|A\|.
\end{equation}
Note that, in view of (\ref{eq:12}) and the fact that $\VERT{K}_{p,\infty} \leq \VERT{K}_{p}$, estimate (\ref{eq:15}) implies that
\begin{equation}
  \label{eq:22}
  n_{A+K}(s) \leq  C^p(p,\infty) \frac{\ln(2)}{2}  \frac{\Phi_p\left(\frac{\|A\|}{s}\right)}{s^p} \VERT{K}_{p}^p, \qquad s>\|A\|.
\end{equation}
Since the optimal values of the constants $C(p,\infty)$ in (\ref{eq:12}) seem to be unknown, we cannot say whether (or for what $p$) the constant appearing in (\ref{eq:22}) is smaller/larger than the constant in (\ref{eq:14}). 
\end{rem}

Our main tool in the proof of Theorem \ref{thm:1} will be the following result, which has first been obtained by Carl \cite{MR619953} for compact operators. Zem{\'a}nek  \cite{MR751086} later observed that it remains correct for general bounded operators as well. 

\begin{thm*}[Carl \cite{MR619953}, Zem{\'a}nek \cite{MR751086}]
  Let the discrete eigenvalues $(\lambda_n)_{n=1}^\infty$ of $A+K$ outside $\overline\md_{r_{ess}}$ be enumerated as described at the beginning of this section. Then  
\begin{equation}
  \label{eq:13}
  |\lambda_n| \leq 2^{\frac{m-1}{2n}} e_m(A+K)
\end{equation}
for all $n,m \in \mn$.
\end{thm*}
Since we emphasized in the introduction that our proof of Theorem \ref{thm:1} does not rely on tools from complex analysis, we should say that also the proof of (\ref{eq:13}) does not rely on such tools. 
\begin{proof}[Proof of Theorem \ref{thm:1}] 
From Carl's inequality (\ref{eq:13}) and the additivity of the entropy numbers (i.e. property (s2) with $n=1$)  we obtain that for all $n,m \in \mn$
$$ |\lambda_n| \leq 2^{\frac{m-1}{2n}} e_m(A+K) \leq 2^{\frac{m-1}{2n}} \left( e_m(K) + \|A\| \right) = 2^{\frac{m-1}{2n}} \left( m^{-\frac 1 p} m^{\frac 1 p} e_m(K) + \|A\| \right).$$
Using the definition of $\|K\|_{p,\infty}$ we thus obtain that 
\begin{equation}
  \label{eq:6}
|\lambda_n| \leq 2^{\frac{m-1}{2n}} \left( m^{-\frac 1 p} \|K\|_{p,\infty}  + \|A\| \right), \qquad n,m \in \mn.  
\end{equation}
To simplify matters, we now introduce a free variable $x>1$ and try to choose $m \in \mn$ such that $2^{\frac{m-1}{2n}}=x$. This equality would be satisfied iff $m=\frac{2n\ln(x)}{\ln(2)} +1$. However, since we don't want to restrict the variable $x$ here the right-hand side need not be an integer. For that reason, we choose instead
$$ m = \left\lceil \frac{2n\ln(x)}{\ln(2)} \right\rceil, \quad \text{i.e. } \quad m \in \mn \text{ and } \frac{2n\ln(x)}{\ln(2)} \leq m < \frac{2n\ln(x)}{\ln(2)} +1.$$
For this choice of $m$ we then obtain from (\ref{eq:6}) that 
\begin{eqnarray*}
  |\lambda_n| &\leq& 2^{\frac{m-1}{2n}} \left( m^{-\frac 1 p} \|K\|_{p,\infty}  + \|A\| \right) 
\leq x \left( \left( \frac{\ln(2)}{2n\ln(x)} \right)^{\frac 1 p} \|K\|_{p,\infty} + \|A\| \right)=:F_n.
\end{eqnarray*}
Now let $s > \|A\|$ and $1<x<s/\|A\|$. Then from the previous inequality we see that if $|\lambda_n|>s$, then also $F_n>s$. Moreover, $n \mapsto F_n$ is decreasing and a short computation shows that 
$$ F_n > s \quad \Leftrightarrow \quad n < \frac{\frac{\ln(2)}{2\ln(x)}\|K\|^p_{p,\infty}}{\left( \frac s x - \|A\|\right)^p}.$$
But this implies that
$$ n_{A+K}(s)=\#\{ n \in \mn : |\lambda_n|>s\} \leq \# \{n \in \mn : F_n > s\} \leq \frac{\frac{\ln(2)}{2\ln(x)}\|K\|^p_{p,\infty}}{\left( \frac s x - \|A\|\right)^p}.$$
Since this is true for all $1<x<\frac s {\|A\|}$ we thus obtain that
\begin{equation}
  \label{eq:20}
  n_{A+K}(s) \leq \frac{\ln(2)}{2} \left( \inf_{1<x<\frac s {\|A\|}} \frac{1}{\ln(x)\left( \frac s x - \|A\|\right)^p} \right) \|K\|_{p,\infty}^p.
\end{equation}
We arrive at (\ref{eq:15}) by computing the infimum in (\ref{eq:20}). We just refer to \cite{MR3296588} (see the computations prior to Theorem 4.2 in that paper) where this computation has been done.  Inequality (\ref{eq:16}) is an immediate consequence of (\ref{eq:15}) and estimate (\ref{eq:18}).
\end{proof}
\begin{rem}
  The idea to apply Carl's inequality (\ref{eq:13}) with an $m$ depending on $n$ has also been used in \cite{MR632627} in a non-perturbative setting.  
\end{rem}
In the following corollary we use the convention from the beginning of this section, regarding the eigenvalues $|\lambda_1| \geq |\lambda_2| \geq \ldots > r_{ess}(A)$ of $A+K$. However, by considering the positive part $(|\lambda_n|-\|A\|)_+:=\max(|\lambda_n|-\|A\|,0)$, we are effectively restricting ourselves to eigenvalues outside $\overline{\md}_{\|A\|}$.
\begin{cor}\label{cor:1}
  Let $0 \neq A \in \bdd(X)$ and $K \in \mathcal S_{p,\infty}^{(e)}(X), p>0$. Then 
  \begin{equation}
    \label{eq:9}
    \left\| \big( \left( |\lambda_n|-\|A\| \right)_+ \big) \right\|_{p+1,\infty}^{p+1}  \leq C_p r(A+K) \|K\|_{p,\infty}^p,
  \end{equation}
with $C_p$ as defined in Theorem \ref{thm:1}.
\end{cor}
\begin{proof}
Using (\ref{eq:1}) we see that
\begin{equation}
  \label{eq:30}
\left\| \big( \left( |\lambda_n|-\|A\| \right)_+ \big) \right\|_{p+1,\infty}^{p+1}= \sup_{r(A+K)>s>\|A\|} \left( (s-\|A\|)^{p+1} n_{A+K}(s) \right).  
\end{equation} 
Now apply estimate (\ref{eq:16}).
\end{proof}
As we have already discussed in the introduction, it is possible that (\ref{eq:9}) is not optimal and that the index $p+1$ on the left-hand side can be replaced by an index $q \in [p,p+1)$. However, the next example shows that even for a rank one operator $K$ the $l_{q,\infty}$-norm of $\big(\left( |\lambda_n|-\|A\| \right)_+\big)_{n \in \mn}$ need not be finite when $q<1$, in sharp contrast to what happens in the compact case (i.e. $A=0$), where this sequence would be in $l_{q,\infty}$ for any $q>0$ (because the entropy numbers of a finite rank operator decay exponentially, see \cite{MR1098497}). 
\begin{ex}\label{ex:1}
  As has been proven in \cite[Example 5.2]{MR3296588}, if $A \in \bdd(l_1(\mn))$ denotes the shift operator then there exists $K \in \bdd(l_1(\mn))$ of rank one such that
$ \big(\left( |\lambda_n|-\|A\| \right)_+ \big)_{n \in \mn}$ is not in $l_1$. In particular, for $q<1$ also $\big( (|\lambda_n|-\|A\|)_+ \big)_{n \in \mn} \notin l_{q,\infty}$ since in this case $l_{q,\infty} \subset l_1$.
\end{ex}
It is an open question if for $K$ of finite rank the sequence $\big( \left( |\lambda_n|-\|A\| \right)_+ \big)_{n \in \mn}$ is always in $l_{1,\infty}$ (Corollary \ref{cor:1} only implies that this sequence is in $l_{p+1,\infty}$ for every $p>0$).\\

Estimate (\ref{eq:16}) has the defect that the right-hand side explodes for $s \to \|A\|$ even in case that $\|A\| > r(A)$ when the left-hand side obviously stays bounded for $s \to \|A\|$. One might hope that an improved estimate of the form
\begin{equation}
  \label{eq:31}
n_{A+K}(s) \leq \tilde C_p  \frac{s} {(s-r(A))^{p+1}} \|K\|_{{p,\infty}}^p, \qquad s>r(A),  
\end{equation}
is true, which does not show this defect. However, the next example shows that, at least in this form, this is not the case.
\begin{ex}
We show that (\ref{eq:31}) does not even hold when the $\mathcal S_{p,\infty}^{(e)}$-norm is replaced by the $\mathcal S_{p,\infty}^{(a)}$-norm (compare (\ref{eq:12})). To this end, let $\hil=\mc^N$ and let 
 $$ A = \left(
  \begin{array}{cccc}
    0 & 1 & &   0\\
      & 0 & \ddots &   \\
      &   & \ddots & 1   \\
    0  &   &        & 0  
  \end{array}\right), \quad 
K =  \left(
  \begin{array}{ccccc}
    0 & 0  & \cdots  & 0 \\
    \vdots  & \vdots    & & \vdots\\
     0 & 0  & \cdots   & 0\\
    1  & 0  & \cdots  & 0 
  \end{array}\right). $$
Then $r(A)=0$ and $\VERT{K}_{p,\infty}^p=1$ for all $p > 0$. Moreover, $A+K$ has $N$ distinct eigenvalues on the unit circle $\partial \md$. Hence, for a fixed $0<s < 1$ we would obtain from (\ref{eq:31}) that
$$ N =n_{A+K}(s) \leq \tilde C_p s^{-p}.$$
Choosing $N$ large enough leads to a contradiction.
\end{ex}

Despite of the previous example we will see that there does indeed exist a way to extend Theorem \ref{thm:1} to an estimate for eigenvalues outside $\overline\md_{r(A)}$. In order to state this extension, we need to introduce the following quantity:
\begin{equation}
  \label{eq:19}
  M_A(r):= \sup_{n \in \mn_0} \frac{\|A^n\|}{r^n}, \qquad r > r(A).
\end{equation}
Some remarks concerning this definition are in order.
\begin{rem}\label{rem:2}
(i) From Gelfand's formula it follows that 
$$1 \leq M_A(r) < \infty, \qquad r>r(A).$$
(ii) $M_A(r)=1$ if $r \geq \|A\|$. In particular, $ M_A \equiv 1$ if $r(A)=\|A\|$.\\
(iii) On the other hand, $M_A(r)$ can grow pretty fast for $r \to r(A)$. To see this, we consider $X=L_p(0,1), 1 \leq p < \infty,$ and the Volterra operator
$$ (Af)(x)=\int_0^x f(t) dt.$$
Then $r(A)=0$ and it was shown in \cite{MR2186094} that there exists $c(p)>0$ such that
$$ \lim_{n \to \infty} n! \|A^n\| = c(p).$$
This shows that 
$$ M_A(r) \approx \sup_{n} \frac 1 {r^n n!}, \quad r>0,$$
so $M_A(r)$ grows faster than any polynomial in $1/r$ for $r \to 0$.

(iv) Since the resolvent norm of $A$ might be easier to compute than the norm of the powers of $A$, the following upper bound might be useful:
  \begin{equation*}
    M_A(r) \leq r \cdot \sup_{|z|=r} \|(z-A)^{-1}\|, \quad r > r(A).
  \end{equation*}
This estimate follows at once from the representation $A^n= \frac 1 {2\pi i} \int_{\partial \md_r} z^n(z-A)^{-1} dz$. 
\end{rem}
The following theorem could be considered the main result of this article. 
\begin{thm}\label{thm:2}
  Let $A \in \bdd(X)$ and $K \in \mathcal S_{p,\infty}^{(e)}(X), p>0$. Then
\begin{equation}\label{eq:x}
 n_{A+K}(s) \leq C_p  \left( \inf_{r(A) < r < s} \frac{ sM_A^p(r)} {(s-r)^{p+1}} \right) \|K\|_{{p,\infty}}^p, \qquad s>r(A).
\end{equation}
In particular,
\begin{equation}\label{eq:y}
 n_{A+K}(s) \leq 2^{p+1} C_p \frac{ s M_A^p(\frac{s+r(A)}{2})} {(s-r(A))^{p+1}} \|K\|_{{p,\infty}}^p, \qquad s>r(A).
\end{equation}
In both cases $C_p$ is as defined in Theorem \ref{thm:1}. 
\end{thm}
\begin{rem}\label{rem:3}  
We note that estimate (\ref{eq:16}) in Theorem \ref{thm:1} and estimate (\ref{eq:x}) in Theorem \ref{thm:2} are actually equivalent. Indeed, Remark \ref{rem:2} (ii) shows that (\ref{eq:16}) follows from (\ref{eq:x}) and in the proof of the theorem below we will prove the reverse implication.
\end{rem}
The idea of proof of Theorem \ref{thm:2} is simple: Replace the operators $A$ and $K$ by similar operators $VAV^{-1}$ and $VKV^{-1}$ (which does not change their spectra) in such a way that the norm of $VAV^{-1}$ is close to the spectral radius of $A$, and then apply Theorem \ref{thm:1}. That this can indeed be done is the content of the following lemma, which uses a construction due to Rota \cite{MR0112040}.
\begin{lemma}   
 Let $A \in \bdd(X)$. Then 
for every $r > r(A)$ there exists a boundedly invertible operator $V_r \in \bdd(X,l_\infty(\mn;X))$ such that 
\begin{equation}
  \label{eq:7}
 \|V_r\|\cdot \|V_r^{-1}\| \leq M_A(r) \quad \text{ and } \quad \|V_r A V_r^{-1}\| \leq r.  
\end{equation}
In particular, we have
\begin{equation}
  \label{eq:32} 
r(A) = \inf\{ \|V AV^{-1}\| : V \in \bdd(X,l_\infty(\mn,X)) \text{ is boundedly invertible}\}.  
\end{equation}
\end{lemma}
\begin{proof}[Proof of the lemma]
Let $r> r(A)$ and set $T:=(1/r)A$. We are going to show that $T$ is similar to a suitable restriction of the shift operator 
$$ U : l_\infty(\mn;X) \to l_\infty(\mn;X), \quad U(x_1,x_2,\ldots)=(x_2,x_3,\ldots).$$
To this end, define $\tilde V : X \to l_\infty(\mn;X)$ by
$$ \tilde V x = (x,Tx,T^2x,\ldots).$$
Then $\tilde V$ is linear, injective and bounded with
$$ \|\tilde V\| \leq \sup_{n \in \mn_0} \|T^n\| = M_A(r).$$
Moreover, it is easy to prove that $Y:=\ran(\tilde V)$ is closed in $l_\infty(\mn;X)$ and that $Y$ is an invariant subspace for $U$. The operator $V=V_r : X \to Y, Vx=\tilde Vx$ is thus boundedly invertible  and a short calculation shows that
$$ T=V^{-1}U|_YV, \quad \|V\| \leq M_A(r) \quad \text{and} \quad \|V^{-1}\|\leq 1.$$
Finally, since $U|_Y$ is a contraction we obtain $\|VAV^{-1}\| = r \|U|_Y\| \leq r$, concluding the proof of (\ref{eq:7}). The validity of (\ref{eq:32}) is a direct consequence of (\ref{eq:7}) and the fact that $A$ and $VAV^{-1}$ have the same spectra. 
\end{proof}
We are now prepared for the 

\begin{proof}[Proof of Theorem \ref{thm:2}]
Let $s> r > r(A)$ and choose $V_r \in \bdd(X,l_\infty(\mn,X))$ as in (\ref{eq:7}). Since the operators $V_r(A+K)V_r^{-1}$ and $A+K$ have the same discrete eigenvalues we thus obtain from  (\ref{eq:16}) and (\ref{eq:7}) that
\begin{eqnarray*}
n_{A+K}(s) &=& n_{V_rAV_r^{-1}+V_rKV_r^{-1}}(s) \leq C_p \frac s {(s-\|V_rAV_r^{-1}\|)^{p+1}} \|V_rKV_r^{-1}\|_{p,\infty}^p\\
&\leq&  C_p \frac s {(s-r)^{p+1}} (\|V_r\|\|V_r^{-1}\|)^p \|K\|_{p,\infty}^p  \leq C_p \frac s {(s-r)^{p+1}} M_A^p(r) \|K\|_{p,\infty}^p. 
\end{eqnarray*}
Since this true for all $r(A)<r<s$ we obtain (\ref{eq:x}). Estimate (\ref{eq:y}) follows from (\ref{eq:x}) by choosing $r=(s+r(A))/2$.
\end{proof}

\section{The Hilbert space case}

In this section we are going to improve our above results in case that $X=\hil$ is a Hilbert space. We start with an estimate for the discrete eigenvalues of $A+K$ in the complement of the numerical range of $A$. Recall that the numerical range $\num(A)$ of $A \in \bdd(\hil)$ is defined as 
$$ \num(A)=\{ \langle Af,f \rangle : f \in \hil, \|f\|=1\}$$
and that by the Toeplitz-Hausdorff theorem it is always a convex set. Moreover, we have
\begin{equation}
  \label{eq:27}
 \sigma(A) \subset  \overline{\num}(A) \subset \overline\md_{\|A\|},
\end{equation}
see, e.g., \cite{MR1417493}.

\begin{convention}
Throughout this section, let $(\lambda_n)_{n=1}^N$ denote an enumeration of $\sigma_d(A+K) \cap (\mc \setminus \num(A))$ (which we assume to be non-empty, so $N \in \mn \cup \{\infty\}$), ordered such that 
$$ d(\lambda_1,\num(A)) \geq d(\lambda_2,\num(A)) \geq \ldots$$ and such that every eigenvalue is counted as many times as its algebraic multiplicity. 
\end{convention}

In the following theorem we consider perturbations $K$ from the approximation number ideal. We recall that in Hilbert spaces we have $\mathcal S_{p,q}^{(e)}(\hil)=\mathcal S_{p,q}^{(a)}(\hil)$. 

\begin{thm}\label{thm:3}
Let $1<p<\infty$ and $0 < q \leq \infty$. Then there exists a constant $C_{p,q} \geq 1$ such that for all $A \in \bdd(\hil)$ and $K \in \mathcal S_{p,q}^{(a)}(\hil)$ we have
\begin{equation}
  \label{eq:5}
  \|\big( d(\lambda_n, \num(A)) \big)_{n=1}^N \|_{p,q} \leq C_{p,q} \VERT{K}_{p,q}.
\end{equation}
\end{thm}
\begin{rem}\label{rem:5}
(i) For $p=q \geq 1$ inequality (\ref{eq:5}) has been proved in \cite{Hansmann11}. In this case the constants $C_{p,p}$ are equal to $1$. While the proof of the above more general theorem is completely analogous to the proof of the $p=q$ case given in \cite{Hansmann11}, we include it for the sake of completeness.\\ 
(ii) We note that (\ref{eq:5}), for the case $p=q$, has been applied to certain operators of mathematical physics, like non-selfadjoint potential perturbations of Schr\"odinger type operators, see, e.g. \cite{MR3443274, frank15, Hansmann11}. With the help of Theorem \ref{thm:3} it should be possible to refine these results, for instance, by using the results of \cite{MR1113389} for estimates on $\mathcal S_{p,q}$-norms of Hille-Tamarkin integral operators. In this context see also \cite{MR1289751}.  
\end{rem}
\begin{proof}[Proof of Theorem \ref{thm:3}] 
There exist complex numbers $(b_{nm})$ and an orthonormal sequence $(e_n)_{n=1}^N$  such that
\begin{equation}
  \label{eq:21}
  (A+K)e_n= b_{n1}e_1 + b_{n2} e_2 + \ldots +b_{nn}e_n \quad \text{and} \quad b_{nn}=\lambda_n,
\end{equation}
see e.g. Lemma 4.1 and the accompanying remark in \cite{b_Gohberg69}. Moreover, by a result of Pietsch \cite[Theorem 2.11.18]{MR917067}, there exists a constant $C_{p,q} \geq 1$ such that
 \begin{equation}
   \label{eq:23}
  \sup \left\{ \left\| (\langle Kx_n,y_n \rangle)_{n=1}^M \right\|_{p,q} \right\} \leq C_{p,q}\VERT{K}_{p,q},
 \end{equation}
 where the supremum is taken over all (finite or infinite) orthonormal sequences $(x_n)_{n=1}^M,(y_n)_{n=1}^M$ in $\hil$. Choosing $x_n=y_n=e_n$ we obtain from (\ref{eq:23}) that
 \begin{equation}
   \label{eq:24}
\left\| (\langle Ke_n,e_n \rangle)_{n=1}^N \right\|_{p,q} \leq  C_{p,q}  \VERT{K}_{p,q}.   
 \end{equation}
But we can use (\ref{eq:21}) to compute
\begin{eqnarray*}
|\langle Ke_n,e_n \rangle| &=& |\langle(A+K) e_n,e_n \rangle - \langle Ae_n,e_n \rangle| \\
&=& |\lambda_n - \langle Ae_n,e_n \rangle| \geq d(\lambda_n,  \num(A)).  
\end{eqnarray*}
Plugging this estimate into (\ref{eq:24}) concludes the proof.
\end{proof}
\begin{rem}
  An inspection of Pietsch's proof of (\ref{eq:23}) shows that
  \begin{equation}
    \label{eq:26}
 C_{p,\infty}= \sup_{n \in \mn} \left( n^{\frac 1 p -1} \sum_{k=1}^n k^{-\frac 1 p } \right), \quad p > 1.    
  \end{equation}
In particular, using the integral test one can show that $C_{p,\infty} \leq  p/(p-1)$.
\end{rem}
\begin{cor}
Let $1<p<\infty$. For $A \in \bdd(\hil)$ and $K \in \mathcal S_{p,\infty}^{(a)}(\hil)$ we have 
\begin{equation}
  \label{eq:25}
  \#\{n : d(\lambda_n,\num(A)) > r\} \leq C^p_{p,\infty} r^{-p} \VERT{K}_{p,\infty}^p, \qquad r>0.
\end{equation}
Here $C_{p,\infty}$ is given by (\ref{eq:26}). 
\end{cor}
\begin{proof}
Use (\ref{eq:1}) applied to $x_n=d(\lambda_n,\num(A))$ and then apply the previous theorem.
\end{proof}
The next corollary, for the case of Hilbert space operators, improves upon the estimates (\ref{eq:16}) and (\ref{eq:x}) . We recall that $M_A(r)=\sup_{n \in \mn_0} \|A^n\|/r^n$.
\begin{cor}
Let $1<p<\infty$. For $A \in \bdd(\hil)$ and $K \in \mathcal S_{p,\infty}^{(a)}(\hil)$ we have  
\begin{equation*} 
 n_{A+K}(s) \leq C^p_{p,\infty}  (s-\|A\|)^{-p} \VERT{K}_{{p,\infty}}^p, \qquad s>\|A\|.
\end{equation*}
More generally,
\begin{equation*}  
 n_{A+K}(s) \leq C^p_{p,\infty}  \left( \inf_{r(A) < r < s} \frac{M_A^p(r)} {(s-r)^p} \right) \VERT{K}_{{p,\infty}}^p, \qquad s>r(A).
\end{equation*}
In both cases $C_{p,\infty}$ is given by (\ref{eq:26}). 
\end{cor}
\begin{proof}
  First, let $s>\|A\|$. Then by (\ref{eq:27}) we obtain that 
\begin{equation*}
 n_{A+K}(s) =\# \{ n : d(\lambda_n,\overline\md_{\|A\|})>s-\|A\|\} \leq \#\{n : d(\lambda_n,\num(A)) > s-\|A\|\},
\end{equation*}
so (\ref{eq:25}) implies that
$$n_{A+K}(s) \leq C^p_{p,\infty} (s-\|A\|)^{-p} \|K\|_{p,\infty}^p, \qquad s>\|A\|.$$
Now in order to prove the second estimate  we just proceed as in the proof of Theorem \ref{thm:2} (or simply consider Remark \ref{rem:3}).      
\end{proof}

 \bibliography{Bibliography} 
 \bibliographystyle{plain}

\end{document}